\documentclass[12pt]{article}
\usepackage{amsmath,amssymb,graphicx,amsthm}
\usepackage{subfigure}
\usepackage{hyperref}

\usepackage[utf8]{inputenc}

\newtheorem{theorem}{Theorem}[section]
\newtheorem{lemma}[theorem]{Lemma}
\newtheorem{corollary}[theorem]{Corollary}
\newtheorem{observation}[theorem]{Observation}

\newtheorem{claim}[theorem]{Claim}
\newtheorem{definition}[theorem]{Definition}

\newcommand\cref[1]{Corollary~\ref{cor:#1}}

\newcommand\oref[1]{Observation~\ref{obs:#1}}

\begin{document}

\title{Two-coloring triples such that in each color class every element is missed at least once}

\author{Bal\'azs Keszegh\thanks{Alfr\'ed R{\'e}nyi Institute of Mathematics and MTA-ELTE Lend\"ulet Combinatorial Geometry Research Group. Research supported by the Lend\"ulet program of the Hungarian Academy of Sciences, under the grant LP2017-19/2017 and by the National Research, Development and Innovation Office -- NKFIH under the grant K 132696.}}

\maketitle
\begin{abstract}
We give a characterization of finite sets of triples of elements (e.g., positive integers) that can be colored with two colors such that for every element $i$ in each color class there exists a triple which does not contain $i$. We give a linear (in the number of triples) time algorithm to decide if such a coloring exists and find one if it does.

We also consider generalizations of this result and an application to a matching problem, which motivated this study. Finally, we show how these results translate to results about colorings of hypergraphs in which the degree of every vertex is $k$ less than the number of hyperedges.
\end{abstract}

\section{Introduction}

For positive integers $k$ and $n$ we are given a finite multiset of $n$ many $k$-tuples\footnote{A $k$-tuple simply denotes a set of size $k$, .e.g., $\{1,2,3\}$ is a $3$-tuple, also referred to as a triple. Note that repetition of elements is not allowed and the order of the elements does not matter.} of characters from an alphabet such that every triple consists of three different characters. From now on a `set of $k$-tuples' always refers to such a finite multiset of $k$-tuples. We also refer to the characters as elements.

Two sets of $k$-tuples is said to be \emph{equivalent} if there is a bijection between their alphabets which induces a bijection between the two sets of $k$-tuples. We do not want to distinguish equivalent sets of $k$-tuples, thus without loss of generality we can assume that the elements are positive integers from $[m]=\{1,\dots,m\}$ for some $m$ and each of these $m$ numbers is present in at least one $k$-tuple. If a $k$-tuple does not contain the element $i$ we say that the $k$-tuple \textit{avoids} $i$ (e.g., $\{1,2,3\}$ avoids $4$ but does not avoid $2$). 

A $c$-coloring\footnote{A $c$-coloring of a set is a mapping from this set to a set of $c$ colors (which may, e.g., be denoted by names like red and blue or by the numbers from $[c]$).} of a set of tuples (of numbers from $[m]$)
is a \textit{nice $c$-coloring} if for each of the $c$ colors and for every $i\in [m]$ there is a tuple of that color that avoids $i$. Similarly, a partial $c$-coloring of a set of tuples (that is, not all tuples need to be colored) is a \textit{nice partial $c$-coloring} if for each of the $c$ colors and for every $i\in [m]$ there is a tuple of that color that avoids $i$.
Notice that a nice partial $c$-coloring can always be extended to a nice $c$-coloring by coloring arbitrarily all the tuples that are uncolored in the partial $c$-coloring.

We are in particular interested in nice two-colorings of triples, that is, our aim is to two-color (with colors red and blue) a set of triples such that for each of the two colors and for every $i\in [m]$ there is a triple of that color that avoids $i$. Our main result is a characterization of the sets of triples that admit a nice (partial) $2$-coloring. 
Section \ref{sec:main} contains the characterization, Section \ref{sec:proof} its proof, while in Section \ref{sec:partial} we give a linear (in the number of triples) time algorithm for finding such a coloring if it exists. We further extend this result, and (without having a characterization) we give an algorithm for finding a nice $c$-coloring for every $c$ and $k$ which runs in linear time (in the number of $k$-tuples).
Section \ref{sec:partial} also considers the existence of nice partial colorings that color only a small number of the $k$-tuples. This research is originally motivated by a real life scheduling problem which can be phrased as a matching problem, this connection is discussed in Section \ref{sec:matching}.

\section{Main results}\label{sec:main}

We are mainly interested in characterizing sets of $k$-tuples that admit a nice $c$-coloring for different values of $c$ and $k$ (even more generally we could have non-uniformly sized tuples). Furthermore, we want efficient algorithms to decide if such a $c$-coloring exists and if yes then find it. 

Irrespective of the size of the tuples, for every $c$ a trivial necessary condition for having a nice $c$-coloring is that all elements must be missed from at least $c$ many tuples, that is, the set of tuples is $c$-fair:

\begin{definition}
	A set of tuples (on elements from $[m]$) is \emph{$c$-fair} if each element (of $[m]$) is missed from at least $c$ many tuples.
\end{definition}

The case $c=1$ is trivial, in this case a set of tuples admits a nice $1$-coloring if and only if every element is missed from at least one tuple (i.e., the set of tuples is $1$-fair).

The case $c=2$ and $k=3$ is already non-trivial. For the existence of a nice two-coloring of the triples it is again a trivial necessary condition that every $i\in [m]$ is avoided by at least two triples (i.e., the set of triples is $2$-fair). During the $9$th Emléktábla Workshop Cechl\'arov\'a \cite{perscomm} asked what are the sufficient conditions for the existence of a nice two-coloring.

For brevity a triple $\{x,y,z\}$ is abbreviated as $xyz$ when it does not lead to confusion (e.g., $\{1,2,3\}$ is written simply as $123$). We explicitly define again the $2$-fair property for a set of triples:

\begin{definition}
	Let $T$ be a set of $n$ triples (of positive integers). The triples containing $i$ are denoted by $T_i$. $T$ is \emph{fair} if for every $i$ there are two triples that are not in $T_i$ (i..e, $T$ is $2$-fair).
\end{definition}

Given a not nice two-coloring, we say that a number $i$ makes it not nice if the triples in one of the color classes all contain $i$.

\begin{definition}	
A set of $n$ triples is called \emph{special} if and only if it contains triples of the following form: $n-3$ copies of the triple $123$ plus three more triples, $1**,2**$ and $3**$, where the $*$'s denote arbitrary numbers different from $1,2,3$. A set of $n$ triples which is not special is called \emph{non-special}.
\end{definition}

Observe that a special set of $n$ triples does not admit a nice two-coloring. 

For $n\le 3$ no set of $n$ triples can have a nice two-coloring as one of the color classes contains at most one triple.

Clearly, for $n\ge 4$ being fair and non-special are both necessary conditions for a set of triples to admit a nice two-coloring. We prove that for $n\ge 6$ these conditions are also sufficient. This was conjectured by Salia (for $n\ge 8$) \cite{perscomm}.

We remark that for $n=4,5$ there exist fair non-special sets of triples that nevertheless do not admit a nice two-coloring. E.g., for $n=4$ the set of triples $\{123,145,245,678\}$ and for $n=5$ the set of triples $\{123,124,134,234,567\}$. As there are only finite many triples for $n=4,5$, we omit to list all which admit a nice two-coloring.

\begin{theorem}\label{thm:main}
A set of $n\ge 6$ triples admits a nice two-coloring if and only if it is fair and non-special.
\end{theorem}

Furthermore, we show a linear (in $n$) time algorithm for any $c$ and $k$:

\begin{theorem}\label{thm:linearck}
	For any fixed $c,k$, given a set of $n$ many $k$-tuples, there is an $O(n)$ time algorithm to check if a nice $c$-coloring exists which also finds one if it exists (the dependence on $c$ and $k$ is hidden in the $O$ notation).
\end{theorem} 

The variants of Theorem \ref{thm:main} and Claim \ref{claim:linear} about partial colorings are stated in Section \ref{sec:partial}.

\smallskip
Graph coloring is a recurring tool in (sport) event scheduling (e.g., \cite{LT,JUBW}). The original motivation of our research is also a (real life) event scheduling problem which can be phrased as a certain matching problem which in turn can be solved using our coloring results. This connection is discussed in detail in Section \ref{sec:matching}.

\subsection{Consequences about coloring hypergraphs}
\smallskip
To put our results in additional context, we phrase our results also as statements about proper and polychromatic coloring certain hypergraphs. A coloring of the vertices of a hypergraph is \emph{proper} if no hyperedge is monochromatic. A $c$-coloring of the vertices is \emph{polychromatic} if every hyperedge contains a vertex with each of the $c$ colors. Notice that for $c=2$ a coloring is proper if and only if it is  polychromatic but for $c\ne 2$ the two conditions differ. 

Given a set $T$ of $n$ triples with elements from $[m]$, let $H_T$ be the multi-hypergraph whose vertices correspond to the triples and for each $i\in [m]$ there is a hyperedge $e_i$ containing exactly those vertices for which the corresponding triple \textit{does not contain} $i$. Note that every vertex is contained in exactly $m-3$ hyperedges. It is easy to see that this mapping $T\rightarrow H_T$ from sets of $n$ triples on elements from $[m]$, to multi-hypergraphs with $n$ vertices and $m$ hyperedges that have all degrees equal to $m-3$, is in fact a bijection. It is also easy to see that a nice two-coloring of the triples of $T$ corresponds to a proper two-coloring of the vertices of $H_T$. With this notation Theorem \ref{thm:main} is equivalent to the following statement:

\begin{theorem}\label{thm:mainhg}
	Given a multi-hypergraph $H$ with $n$ vertices and $m$ hyperedges such that every vertex has degree $m-3$, $H$ admits a proper two-coloring if and only if every hyperedge of $H$ has size at least $2$ and $H$ is triangle-free\footnote{A hypergraph $H$ is triangle-free if in $H$ there are no three vertices $a,b,c$ for which $\{a,b\},\{a,c\},\{b,c\}$ are hyperedges (of size $2$) in $H$.}. 
\end{theorem} 

Notice that these conditions are trivially necessary, as the existence of a hyperedge of size at most $1$ or the existence of a triangle immediately prevents the hypergraph from admitting a proper coloring. Theorem \ref{thm:main} implies that these simple conditions are also sufficient. Theorem \ref{thm:linearck} can be also translated to the language of hypergraphs:

\begin{theorem}\label{thm:linckhg}
	For any fixed $c,k$, given a multi-hypergraph $H$ with $n$ vertices and $m$ hyperedges such that every vertex has degree $m-k$, there is an $O(n)$ time algorithm\footnote{We assume that the multi-hypergraph is stored such that for each vertex the $k$ hyperedges \emph{not} containing this vertex is given.} to check if $H$ admits a polychromatic $c$-coloring which also finds one if it exists.
\end{theorem} 

In general it is well known that proper two-colorability of a hypergraph is $NP$-complete \cite{lovasz,garey}. In contrast to this, Theorem \ref{thm:mainhg} states that there is even a simple characterization of those hypergraphs in which the degree of every vertex is $3$ less than the number of hyperedges and that are proper two-colorable. Furthermore, for every $c$ and $k$, if in a hypergraph the degree of every vertex is exactly $k$ less than the number of hyperedges, Theorem \ref{thm:linckhg} gives a linear time algorithm (in terms of the number of vertices) to decide if the hypergraph admits a polychromatic $c$-coloring (which is equivalent to a proper $2$-coloring when $c=2$) and also finds the coloring when it exists.

\section{Proof of the characterization} \label{sec:proof}
The proof of Theorem \ref{thm:main} is based on the following two lemmas.

\begin{lemma}\label{lem:tight}
	If in a fair set of non-special triples there are at least three numbers that appear exactly $n-2$ times then the set admits a nice two-coloring.
\end{lemma}
\begin{proof}
	There are three numbers, wlog. the numbers $1,2,3$, that appear $n-2$ times. This implies that there are $n-6$ triples of the form $123$, removing these triples we get a set $T$ of $6$ triples in which the numbers $1,2,3$ appear exactly $4$ times. If we can find a nice two-coloring of these $6$ triples than an arbitrary extension of this coloring to the original set gives a nice two-coloring of the original set. We have that $|T_1\cap T_2|,|T_1\cap T_3|,|T_2\cap T_3|\ge 2$.
	
	\begin{enumerate}
		\item There exist two numbers $i,j$ out of $1,2,3$, for which $T_i= T_j$.
		
		Then wlog. there are three cases, listed in Table \ref{tab:4a}.
		\begin{table}[h]
			\centering
			\begin{tabular}{lllllllllllllllllll}
				1 & 2 & 3 &  &  &  &  &  & 1 & 2 & 3 &  &  &  &  &  & 1 & 2 & 3 \\
				1 & 2 & 3 &  &  &  &  &  & 1 & 2 & 3 &  &  &  &  &  & 1 & 2 & 3 \\
				1 & 2 & 3 &  &  &  &  &  & 1 & 2 & 3 &  &  &  &  &  & 1 & 2 & * \\
				1 & 2 & 3 &  &  &  &  &  & 1 & 2 & * &  &  &  &  &  & 1 & 2 & * \\
				* & * & * &  &  &  &  &  & * & * & 3 &  &  &  &  &  & * & * & 3 \\
				* & * & * &  &  &  &  &  & * & * & * &  &  &  &  &  & * & * & 3
			\end{tabular}
			\caption{Case 1.}
			\label{tab:4a}
		\end{table}
		
		In all three cases we color the first, the third and the last triple red and the rest of the triples blue to get a nice coloring.
		
		
		\item $|T_1\cap T_2|=|T_1\cap T_3|=|T_2\cap T_3|=2$.
		
		\begin{table}[h]
			\centering
			\begin{tabular}{lll}
				1 & 2 & * \\
				1 & 2 & * \\
				1 & * & 3 \\
				1 & * & 3 \\
				* & 2 & 3 \\
				* & 2 & 3
			\end{tabular}
			\caption{Case 2.}
			\label{tab:4b}
		\end{table}
		
		For this case see Table \ref{tab:4b}. 
		If there is no number that appears more than $4$ times then it is easy to see that there exists a nice two-coloring of these $6$ sets. If there exists an $i$ that appears $5$ times, then in the original (fair) set there was an additional triple $123$, in which case it is easy to see that there exists a nice two-coloring of these $7$ sets. Finally, if there exists an $i$ that appears $6$ times, then in the original (fair) set there were two additional triples $123$, and then it is again easy to see that there exists a nice two-coloring of these $8$ sets.
						
		\item $2\le |T_1\cap T_2|,|T_1\cap T_3|,|T_2\cap T_2|\le 3$ and not all of $|T_1\cap T_2|,|T_1\cap T_3|,|T_2\cap T_2|$ are equal. 
		
		This implies that wlog. $|T_1\cap T_2|=2$ while $|T_1\cap T_3|=3$. Then wlog. there are two cases, listed in Table \ref{tab:4c}.
		
		\begin{table}[h]
			\centering
			\begin{tabular}{lllllllllll}
				1 & * & 3 &  &  &  &  &  & 1 & * & 3 \\
				1 & * & 3 &  &  &  &  &  & 1 & * & * \\
				1 & 2 & 3 &  &  &  &  &  & 1 & 2 & 3 \\
				1 & 2 & * &  &  &  &  &  & 1 & 2 & 3 \\
				* & 2 & 3 &  &  &  &  &  & * & 2 & 3 \\
				* & 2 & * &  &  &  &  &  & * & 2 & *
			\end{tabular}
			\caption{Case 3.}
			\label{tab:4c}
		\end{table}
		
		In the first case if there is no number that appears more than $4$ times then it is easy to see that there exists a nice two-coloring of these $6$ sets. If there exists an $i$ that appears $5$ times, then in the original (fair) set there was an additional triple $123$, and then it is easy to see that there exists a nice two-coloring of these $7$ sets. Finally, no number can appear $6$ times as the third triple is $123$.
		
		In the second case we color the first, third and the last triple red and the rest of the triples blue to get a nice coloring.
		
		\item $|T_1\cap T_2|=|T_1\cap T_3|=|T_2\cap T_3|=3$.
		
		Then wlog. there are two cases, listed in Table \ref{tab:4d}.
		
		\begin{table}[h]
			\centering
			\begin{tabular}{lllllllllll}
				1 & 2 & 3 &  &  &  &  &  & 1 & 2 & 3 \\
				1 & 2 & 3 &  &  &  &  &  & 1 & 2 & 3 \\
				1 & 2 & 3 &  &  &  &  &  & 1 & 2 & * \\
				1 & * & * &  &  &  &  &  & 1 & * & 3 \\
				* & 2 & * &  &  &  &  &  & * & 2 & 3 \\
				* & * & 3 &  &  &  &  &  & * & * & *
			\end{tabular}
			\caption{Case 4.}
			\label{tab:4d}
		\end{table}
		
		In the first case we have a special set of triples, a contradiction. In the second case we color the first and last triple red and the rest of the triples blue to get a nice coloring.
	\end{enumerate}

\end{proof}
\begin{lemma}\label{lem:six}
A fair non-special set of $6$ triples admits a nice two-coloring.
\end{lemma}

\begin{proof}
	We again split the problem into a few cases.
	
	The $6$ triples of the set together contain $18$ numbers (with multiplicities). As the set is fair, every number appears at most $4$ times.
	We distinguish cases based on how many numbers appear exactly $4$ times.
	
	\begin{enumerate}
		\item
		No number appears $4$ times. In this case there are at most $18/3=6$ numbers that appear $3$ times. We consider only two-colorings where both color classes contain $3-3$ triples (we call such colorings \emph{balanced}) and prove that at least one of them is nice. There are $10$ such colorings (we do not distinguish pairs of colorings with switched colors). A number that appears $3$ times makes exactly one balanced coloring not nice, thus there are at least $10-6=4$ nice two-colorings.
		\item
		Exactly one number, wlog. the number $1$, appears $4$ times. We again consider only the $10$ balanced two-colorings and prove that at least one of them is nice. There are at most $\lfloor (18-4)/3\rfloor=4$ numbers which appear $3$ times, each of them makes one coloring not nice. Number $1$ makes $4$ colorings not nice, thus altogether there are at least $10-4-4=2$ nice balanced two-colorings.
		\item
		Exactly two numbers, wlog. $1$ and $2$, appear $4$ times. 
		
		In this case $|T_1\cap T_2|\ge 2$, and there are at most $\lfloor 18-2\cdot 4\rfloor=3$ numbers which appear in exactly three triples. We distinguish some subcases:
		\begin{enumerate}
			\item $|T_1\cap T_2|=4$, i.e., $T_1=T_2$. 
			
			Again we consider only the balanced two-colorings. $1$ and $2$ both make the same four of them not nice while the at most $3$ numbers which appear in exactly three triples make $3$ of them not nice, so still there are at least $10-4-3=3$ nice balanced two-colorings.
			
			\item $|T_1\cap T_2|=3$.
			
			In this case we have triples $12*,12*,12*,1**,2**,***$ where $*$ are numbers different from $1,2$.
			Again we consider only the balanced two-colorings. The numbers $1$ and $2$ together make $7$ of them not nice while the at most $3$ numbers which appear in exactly three triples make $3$ of them not nice. Assume first that there is some coincidence among these not nice balanced colorings, then there is at least $10-7-3+1=1$ nice balanced two-coloring. 
			
			Now assume that all these $10$ not nice balanced colorings are different, then there are numbers, wlog. $3,4,5$ such that $|T_3|=|T_4|=|T_5|=3$ and $|T_i\cap T_j|=2$ for every $i=1,2$ and $j=3,4,5$. This implies $|T_1\cap T_2\cap T_j|\ge 1$ for $j=3,4,5$.
			Then the first three triples must be $123,124,125$. To have $|T_1\cap T_j|=2$ for $j=3,4,5$ we need that the fourth triple containing $1$ contains all of $3,4,5$, a contradiction.
			
			\item $|T_1\cap T_2|=2$.
			
			Again we consider only the balanced two-colorings. The numbers $1$ and $2$ together make $6$ of them not nice while the at most $3$ numbers which appear in exactly three triples make $3$ of them not nice, so still there are at least $10-6-3=1$ nice balanced two-colorings.
		\end{enumerate}

	\item At least three numbers, wlog. the numbers $1,2,3$, appear $4$ times.
	
	This case follows from Lemma \ref{lem:tight}.
\end{enumerate}

\end{proof}

We introduce one more notation and then we are ready to prove Theorem \ref{thm:main}.

\begin{definition}
	Let $T$ be a set of $n$ triples (of positive integers). $T$ is \emph{reducible} if we can delete a triple from it such that the remaining set of triples is fair, otherwise it is \emph{irreducible}.
\end{definition}

Note that a reducible set of triples is by definition necessarily fair. 

\begin{proof}[Proof of Theorem \ref{thm:main}]
	We have seen earlier that the conditions are necessary, so we want to prove that they are also sufficient. That is, we want to find a nice two-coloring of a fair non-special set $T$ of $n\ge 6$ triples.
	
	If $T$ is reducible then we delete one of the triples such that the remaining set is still fair. We keep doing this until we get an irreducible set $T'$ or a set $T'$ with exactly $6$ triples. 
	
	\begin{enumerate}
		\item $T'$ is non-special. 
	
	If $T'$ has $n'=6$ triples then by Lemma \ref{lem:six} we get a nice two-coloring of $T'$. Otherwise $T'$ is irreducible. 
	
	If $T'$ is irreducible, deleting an arbitrary triple $t$ makes the set not fair, thus there is a number (wlog. the number $1$) which appears $n-2$ times (and does not appear in $t$). Next, deleting a triple $t'$ which contains $1$ would make the set not fair, thus there is a number which appears $n-2$ times and does not appear in $t$, thus this number is different from $1$, wlog. $2$. Finally, as $n\ge 6$, there is a triple $t''$ in which $1$ and $2$ both appear. Deleting $t''$ would also make the set not fair thus there is a number different from $1$ and $2$, wlog. $3$, which also appears $n-2$ times. Thus, there are three numbers that appear $n-2$ times in the fair set of triples $T'$, so by Lemma \ref{lem:tight} we get a nice two-coloring of $T'$.
	
	In both cases, the nice two-coloring of $T'$ can be extended arbitrarily to a nice two-coloring of $T$.
	
	\item $T'$ is special. 
	
	$T'$ is then a special fair set of $n'\ge 6$ triples. Wlog. $T'$ consists of $n'-3\ge 3$ triples of the form $123$ and three triples, $t_1=1**,t_2=2**,t_3=3**$ (where $*$ denote arbitrary numbers different from $1,2,3$).
	Now we are interested in the triples that were deleted during the process. Recall that $T$ was a non-special set, thus we must have deleted at least one triple $t$ which is not of the form $123$, thus $t$ avoids at least one of $1,2,3$. Assume wlog. that $t$ avoids $1$. Color $t, t_1$ and one triple $123$ with color red. Color the rest of the triples (including $t_2,t_3$ and another triple $123$) blue, it is easy to check that this coloring is nice, as required.
\end{enumerate}	
\end{proof}

We mention that in Theorem \ref{thm:main} we use Lemma \ref{lem:six} only on irreducible sets.

\section{Algorithms and partial colorings}\label{sec:partial}

For general $c$ and $k$, if a nice (partial) $c$-coloring exists of $k$-tuples, then in each color class we can choose at most $k+1$ triples such that coloring these at most $c(k+1)$ $k$-tuples (the rest of the triples can remain uncolored) already has the property of a nice partial $c$-coloring. Indeed, for each color we can choose an arbitrary $k$-tuple of that color, then using that the coloring is nice, we can choose at most $k$ other $k$-tuples of that color avoiding each element in this $k$-tuple, together these at most $c$ times $k+1$ many $k$-tuples are as required. 

\begin{observation}\label{obs:partialize}
	If a nice (partial) $c$-coloring exists of a set of $k$-tuples then there is also a nice partial $c$-coloring of the $k$-tuples which uses all colors at most $k+1$ times and the original coloring is an extension of this coloring. Moreover, from each color class of the original $c$-coloring we can fix one $k$-tuple which remains colored in the new nice partial $c$-coloring (with the same color as in the original coloring). Such a nice partial $c$-coloring can be found easily in linear time if the nice (full) $c$-coloring is given.
\end{observation}

From these using Theorem \ref{thm:main} we get the following:

\begin{corollary}\label{cor:mainpartial}
	Given a set of $n\ge 6$ triples, a nice partial $2$-coloring that colors at most $4$ triples with each of the two colors exists if and only if the set of triples is fair and non-special.
\end{corollary} 

\oref{partialize} implies that there is a $O(n^{c(k+1)})$ time algorithm to check for a set of $n$ $k$-tuples if a nice $c$-coloring exists and find one if it exists. Indeed, it is enough to check the $O(n^{c(k+1)})$ many partial colorings that color $k+1$ $k$-tuples with each color whether any of them is a nice partial $c$-coloring (and if yes, extend it arbitrarily to a nice $c$-coloring). Note that checking any one of these colorings whether it is nice can be done in constant time (dependent on $c$ and $k$). 

In case $c=1$ we have seen that a set of $k$-tuples has a nice $1$-coloring if and only if it is $1$-fair which can be easily checked in linear time in $n$. If it is $1$-fair then coloring all $k$-tuples with the unique color is a nice coloring. Also, we can easily find in linear time in $n$ a subset of at most $k+1$ many $k$-tuples such that coloring only these is a nice partial $1$-coloring.

In case $c=2$ and $k=3$ the above argument gives that in time $O(n^8)$ we can check if a nice $2$-coloring exists of a set of triples and if yes then also find one. For this case we can improve considerably this naive algorithm. Checking that a set of $n$ triples is fair and non-special can be done easily in linear time in $n$. Indeed, being special is very easy to check while testing if a set of triples is $2$-fair, one can choose two arbitrary triples, and only check if the elements present in these two triples are avoided by at least two other triples, as these two triples both avoid all other elements.

This and Theorem \ref{thm:main} implies that there is a linear time algorithm to check if a nice $2$-coloring of a set of triples exists. This does not immediately give an algorithm to also find such a coloring. Next we show how the characterization leads to a linear time algorithm for also finding a nice $2$-coloring when it exists.

\begin{claim}\label{claim:linear}
	Given a set of $n$ triples, there is an $O(n)$ time algorithm to check if a nice $2$-coloring exists and find one if it exists.
\end{claim} 

\begin{proof}
For $n\le 5$ we can check every $2$-coloring in constant time. Given a set $T$ of $n\ge 6$ triples, checking if a set of triples is fair and non-special can be easily done in linear time. If these conditions hold, then we know that there exists a nice $2$-coloring (and otherwise it does not). Assuming that the set of triples has both of these properties, our aim is to find a constant size subset of the triples which already has both of the properties.

In order to do that, take two arbitrary triples, $e$ and $f$. As the set is fair, for each element appearing on $e$ or $f$, in linear time we can find two triples that avoid this element. Altogether $e$ and $f$ has at most $6$ different elements, and thus we find at most $12$ triples which together with $e$ and $f$ form the set $T'$ (with size at most $14$). We can also assume that $T'$ has at least $6$ triples as otherwise we add to it arbitrarily some further triples so that this holds. We claim that $T'$ is fair. Indeed, by our construction for each element in $e$ or $f$ there are at least two triples in $T'$ which avoid it, while for every other element both $e$ and $f$ avoid that element. Now we check if $T'$ is non-special, which can be checked in constant time. If yes, we are done. On the other hand, if $T'$ is special then there is a unique triple $g$ that occurs at least $3$ times in $T'$, this can be identified in constant time. As $T$ is not special, in linear time we can find an additional triple from $T\setminus T'$ which is different from $g$, adding this to $T'$ makes it non-special (and it remains to be fair).

Finally, having found a constant size (at most $15$) subset $T'$ which is fair and non-special, we can check in constant time all its two-colorings to find one which is nice. This is also a nice partial two-coloring of $T$, which can be extended arbitrarily (in linear time) to a nice two-coloring of $T$.

Altogether the algorithm takes $O(n)$ time, as required.
\end{proof}

In fact there is a linear time algorithm for every $c,k$. Note that for general $c,k$ we do not have a characterization and so the algorithm is based only on the fact that it is enough to find a small partial coloring, this is stated by Theorem \ref{thm:linearck}.

\begin{proof}[Proof of Theorem \ref{thm:linearck}]
	We fix some $c$ and $k$ which are considered to be constants and we are given a set $T$ of $n$ many $k$-tuples. The proof idea is to reduce the size of the problem, that is, we will create a constant size set $R$ of $k$-tuples such that $T$ admits a nice $c$-coloring if and only if $R$ does, moreover, given a nice $c$-coloring of $R$, we can find a nice partial $c$-coloring of $T$ in constant time.
	
	Fix an arbitrary subset $S$ of $s=(k+1)(c-1)+1$ many $k$-tuples. If a nice $c$-coloring of $T$ exists then by \oref{partialize} also a nice partial $c$-coloring exists which colors at most $k+1$ sets with each color. In this partial $c$-coloring for some integer $i$ ($0\le i\le c$) we have that among the $k$-tuples in $S$ there are at least $i$ colors present and also there are at least $c-i$ uncolored sets in $S$. We can easily extend this partial coloring to a coloring such that all colors are present on $S$ (for each color missing on $S$ we color one uncolored $k$-tuple of $S$ with this color). Summarizing, if there exists a nice $c$-coloring then there exists also a nice $c$-coloring such that all colors appear on $S$.
	
	Notice that $s$, the size of $S$, is a constant. We make a list $E'$ of the at most $ks$ elements that appear in the sets of $S$. From now on during the algorithm whenever we see an element not in $E'$, we replace it with a dummy element $*$ (it can happen that a $k$-tuple now contains several $*$'s but it will cause no problems). By this our alphabet is essentially reduced to size at most $ks+1$ (the elements in $E'$ plus $*$), and we get the set of $k$-tuples $T_*$ on this alphabet. Note that the $k$-tuples of $T_*$ are in a natural bijection with $k$-tuples of $T$, which, given a (partial) coloring of $T_*$, defines a partial coloring of $T$.
	
	\begin{lemma}
		A nice partial $c$-coloring of $T_*$ is also a nice partial $c$-coloring of $T$. On the other hand, if $T$ admits a nice partial $c$-coloring then $T_*$ also admits a nice partial $c$-coloring.
	\end{lemma}
	
	\begin{proof}
		Clearly, by definition of a nice coloring, if we can find a nice $c$-coloring of $T_*$, then the same coloring is also a nice $c$-coloring of the original set of $k$-tuples (indeed, merging elements just makes our task harder). 
		
		On the other hand we have seen that if $T$ admits a nice partial $c$-coloring then it also admits one in which on $S$ all colors appear. We claim that this coloring is also a nice partial $c$-coloring of $T_*$. The property of a nice coloring requires for each color and each element that there is a $k$-tuple with this color avoiding this element. As we did not merge elements in $E'$, this remains true for every element in $E'$ and every color. Also, it is true for $*$ and every color because for each color any $k$-tuple in $S$ with this color avoids $*$, as required.
	\end{proof}
	
	This lemma shows that it is enough to find a nice partial $c$-coloring of $T_*$. If it does not exist, then $T$ does not have a nice partial $c$-coloring. On the other hand, if it exists, then it is also a nice partial $c$-coloring of $T$. Thus, from now on we restrict our attention to $T_*$.
	
	Observe that in a nice partial $c$-coloring, if some $k$-tuples contain the same elements and get the same color, then by uncoloring all but one of them we still get a nice partial $c$-coloring.
	
	From constant many elements (that is, $ks+1$) there are only constant many different $k$-tuples that can be generated. We go through the set $T_*$ of $k$-tuples one-by-one and if we already kept $c$ copies of the pending $k$-tuple, then we throw it away, otherwise we keep it. This process can be done in $O(n)$ time, at the end we are left with a set $R$ of constant many $k$-tuples, as each different $k$-tuple generated from the $ks+1$ elements has multiplicity at most $c$. By our previous observation, if $T_*$ has a nice partial $c$-coloring then $R$ also has one, as in each color class every type of $k$-tuples needs to be used at most once, and so in all colors together at most $c$ times.
	
	Summarizing, as we promised at the beginning of the proof, we have defined a constant size set $R$ of $k$-tuples which admits a nice (partial) $c$-coloring if and only if $T_*$ does which by the lemma is further equivalent with $T$ admitting a nice (partial) $c$-coloring. Moreover, if such a coloring exists of $R$ then the same coloring is nice for $T_*$ and by the lemma also for $T$.
	
	As $R$ has constant size, we can brute force check in constant time if it admits a nice $c$-coloring and if it does then we can use that coloring to get a nice partial $c$-coloring of $T$ (which can be easily extended to a $c$-coloring of $T$ in linear time).
	
	Altogether the algorithm takes $O(n)$ time, as required.
\end{proof}

As we stated earlier, we can easily uncolor (in linear time) some $k$-tuples in a nice $c$-coloring such that we get a nice partial $c$-coloring in which all colors are used at most $k+1$ times. Claim \ref{claim:linear} thus implies the following:

\begin{corollary}\label{cor:linearpartialck}
	For any fixed $c,k$, given a set of $n$ many $k$-tuples, there is an $O(n)$ time algorithm to check if a nice partial $c$-coloring exists which uses every color at most $k+1$ times, and which finds one if it exists (the dependence on $c$ and $k$ is hidden in the $O$ notation).
\end{corollary}

\section{A matching problem application}\label{sec:matching}
Here we discuss the real life problem that motivated our research, the matching problem it translates to and how these are connected to our results, as it was presented by Cechl\'arov\'a in the $9$th Emléktábla Workshop Booklet \cite{9EWbooklet} and communicated to me by Jankó \cite{perscomm}. It is about the International Young Physicists' Tournament (IYPT), sometimes referred to as `Physics World Cup', a team-oriented scientific competition between secondary school students. The real-world setup is slightly different from the model regarded here, we only restrict our attention to the model relevant for us.

We are given $n$ teams, each chooses in advance $3$ problems to his portfolio (out of a given set of $m$ problems). The teams need to be split into groups of $3$ or $4$ and in each group there are $3$ rounds, and in each round each team of the group presents a problem. It is required that no problem is presented twice within a group in the same round. We are interested in finding conditions and algorithms to see if such a grouping is possible.

In a group let us represent the teams and problems as the vertices of a bipartite graph, a problem is connected to a team if it is in its portfolio. In particular, every team has degree $3$. It is easy to see that the problem is equivalent to splitting the teams into groups of $3$ and $4$ and in each group splitting (in other words, coloring) the edges incident to the teams into $3$ matchings. By König's Line Coloring Theorem this can be done if and only if all degrees are at most $3$ in the subgraph of the edges incident to the teams of a given group. This trivially holds for the degrees of the teams, for the degrees of the problems this means that no problem is present in the portfolio of more than $3$ teams in the group.

In groups of size $3$ this trivially holds, thus only groups of size $4$ may cause an issue. If $n$ is divisible by $3$ then we do not need such groups, we only need that $n\ge 3$. If $n\equiv 1\mod 3$ then we need that $n\ge 4$ and there needs to be one group of size $4$, which is exactly the partial coloring problem for $c=1$, where the $m$ problems correspond to the elements of $[m]$, the $n$ teams to $n$ triples (a team corresponds to a triple containing the problems choosen by this team) and the unique group of size $4$ to the color class of a partial $1$-coloring. For that we have seen that the trivial necessary and sufficient condition is that the set of triples is $1$-fair. Finally, if $n\equiv 2\mod 3$ then we need $n\ge 8$ to be able to split $n$ into sets of size $3$ and $4$. In this case we need two groups of size $4$, which is exactly the partial coloring problem for $c=2$, where the $m$ problems correspond to the elements of $[m]$, the $n$ teams to $n$ triples and the groups of size $4$ to the two color classes of a partial $2$-coloring. Corollary \ref{cor:mainpartial} implies that the necessary and sufficient condition in this case is that the set of triples is $2$-fair and non-special (note that Corollary gives a coloring which uses both colors at most $4$ times, but this can easily be extended to a coloring which uses both colors exactly $4$ times). Thus, we have solved all cases, furthermore, checking the existence of and finding such a coloring can be done in linear time by Corollary \ref{cor:linearpartialck}.

\bigskip
\noindent \textbf{Acknowledgement}

The author is grateful to the organizers and participants of the 9th Emléktábla Workshop, in particular to Katar\'ina Cechlárová, Zsuzsanna Jankó and Nika Salia \cite{perscomm} for their helpful comments. The author also thanks an anonymous reviewer for his comments.


\begin{thebibliography}{99}

\bibitem{9EWbooklet}
9th Emléktábla Workshop Booklet, \url{https://www.renyi.hu/~emlektab/index.html}

\bibitem{JUBW} T. Januario, S. Urrutia, C. C. Ribeiro, and D. de Werra, Edge coloring: A natural model for sports scheduling, European Journal of Operational Research {\bf 254(1)} (2016), 1--8.

\bibitem{perscomm}
K. Cechlárová, Zs. Jankó, N. Salia, personal communication

\bibitem{LT} R. Lewis  and  J.Thompson, On the application of graph colouring techniques in round-robin sports scheduling, Computers \& Operations Research {\bf 38(1)} (2011), 190--204.

\bibitem{garey}
Michael R. Garey and David S. Johnson, Computers and Intractability: A Guide to the Theory of NP-Completeness, W. H. Freeman \& Co.(1979), ISBN 0716710447

\bibitem{lovasz}
László Lovász, Coverings and colorings of hypergraphs, Proc. 4th Southeastern Conf. on Comb., Utilitas Math. (1973), 3--12.
\end{thebibliography}
\end{document}